\documentclass[12pt]{article}
\usepackage{amsthm}
\usepackage{amsmath}
\usepackage{amsfonts}
\usepackage{graphicx}
\usepackage{latexsym}
\usepackage{amssymb}

\DeclareMathAlphabet{\mathbfsl}{OT1}{ppl}{b}{it} 

\newcommand{\Strut}[2]{\rule[-#2]{0cm}{#1}}

\newcommand{\uuu}{\mathbfsl{u}}
\newcommand{\vvv}{\mathbfsl{v}}

\newcommand{\deff}{\mbox{$\stackrel{\rm def}{=}$}}
\newcommand{\Span}[1]{{\left\langle {#1} \right\rangle}}

\newcommand{\sbinom}[2]{\left[ \begin{array}{c} #1 \\ #2 \end{array} \right] }
\newcommand{\sbinomq}[2]{\begin{scriptsize} \sbinom{#1}{#2}_q \end{scriptsize}}
\newcommand{\sbinomtwo}[2]{\begin{scriptsize} \sbinom{#1}{#2}_2 \end{scriptsize}}

\newcommand{\field}[1]{\mathbb{#1}}
\newcommand{\Z}{\field{Z}}
\newcommand{\N}{\field{N}}
\newcommand{\F}{\field{F}}

\newcommand{\dP}{\field{P}}
\newcommand{\dS}{\field{S}}
\newcommand{\T}{\field{T}}
\newcommand{\V}{\field{V}}

\newcommand{\cC}{{\cal C}}

\newcommand{\cG}{{\cal G}}

\newcommand{\B}{{\mathbb B}}
\newcommand{\C}{{\mathbb C}}
\newcommand{\CMRD}{\C^{\textmd{MRD}}}

\newtheorem{theorem}{Theorem}
\newtheorem{lemma}{Lemma}
\newtheorem{remark}{Remark}
\newtheorem{const}{Construction}
\newtheorem{cor}{Corollary}

\newtheorem{example}{Example}

\begin{document}

\bibliographystyle{plain}

\title{
\begin{center}
Covering of Subspaces by Subspaces
\end{center}
}
\author{
{\sc Tuvi Etzion}\thanks{Department of Computer Science, Technion,
Haifa 32000, Israel, e-mail: {\tt etzion@cs.technion.ac.il}.}}

\maketitle

\begin{abstract}
Lower and upper bounds on the size of a covering of subspaces in
the Grassmann graph $\cG_q(n,r)$ by subspaces from the Grassmann
graph $\cG_q(n,k)$, $k \geq r$, are discussed. The problem is of interest from
four points of view: coding theory, combinatorial designs,
$q$-analogs, and projective geometry.
In particular we examine
coverings based on lifted maximum rank distance codes, combined with
spreads and a recursive construction. New constructions are
given for $q=2$ with $r=2$ or $r=3$.
We discuss the density for some of these coverings.
Tables for the best known coverings, for
$q=2$ and $5 \leq n \leq 10$, are presented.
We present some questions concerning possible constructions
of new coverings of smaller size.
\end{abstract}

\vspace{0.5cm}


\footnotetext[1] { This research was supported in part by the Israeli
Science Foundation (ISF), Jerusalem, Israel, under
Grant 10/12.}

\newpage
\section{Introduction}

Let $\F_q$ be a finite field with $q$ elements. For given integers
$n \geq k \geq 0$, let $\cG_q(n,k)$ denote the set of all
$k$-dimensional subspaces of $\F_q^n$. $\cG_q(n,k)$ is often
referred to as Grassmannian. It is well known that
$$ \begin{small}
| \cG_q (n,k) | = \sbinomq{n}{k}
\deff \frac{(q^n-1)(q^{n-1}-1) \cdots
(q^{n-k+1}-1)}{(q^k-1)(q^{k-1}-1) \cdots (q-1)}
\end{small}
$$
where $\sbinomq{n}{k}$ is the $q-$\emph{ary Gaussian
coefficient}~\cite{vLWi}.

A {\it code} $\C$ over the Grassmannian is a subset of
$\cG_q(n,k)$. In recent years there has been an increasing
interest in codes over the Grassmannian as a result of their
application to error-correction in random network coding as was
demonstrated by Koetter and Kschischang~\cite{KoKs}. But, the
interest in these codes has been also before this application,
since these codes are $q$-analogs of constant weight codes. The
well-known concept of $q$-analogs replaces subsets by subspaces of
a vector space over a finite field and their orders by the
dimensions of the subspaces. In particular, the $q \text{-analog}$
of a constant weight code in the Johnson space is a constant
dimension code in the Grassmannian space. $q$-analogs of various
combinatorial objects are well known~\cite[pp. 325-332]{vLWi}.
Design theory is a well studied area in combinatorics related to
coding theory. Related to constant dimension codes are $q$-analogs
of block designs. $q$-analogs of $t$-designs were studied in
various papers and
connections~\cite{BKL,Ito,MMY,ScEt,Suz1,Suz2,Tho1,Tho2}. Only some
of these designs had also some interest in coding
theory~\cite{AAK,ScEt}. These designs are known as Steiner
structures which are $q$-analogs of Steiner systems.

A {\it Steiner structure} $\dS_q(r,k,n)$ is a collection $\dS$ of
elements from $\cG_q(n,k)$ such that each element from
$\cG_q(n,r)$ is contained in exactly one element of $\dS$.

The condition that ``each element from $\cG_q(n,r)$ is contained in
exactly one element of~$\dS$" can be relaxed. If ``each element
from $\cG_q(n,r)$ is contained in at most one element of~$\dS$"
then the structure is a $q$-packing design better known as a
constant dimension code or a Grassmannian code. These codes were
considered in many papers in the last five years,
e.g.~\cite{EtSi09,EtSi11,EtVa09,KoKs,KoKu08,SKK08}.

A {\it $q$-covering design} $\C_q(n,k,r)$ is a collection $\dS$ of
elements from $\cG_q(n,k)$ such that each element of $\cG_q(n,r)$
is contained in at least one element of $\dS$.

Let $\cC_q(n,k,r)$ denote the minimum number of subspaces in a
$q$-covering design $\C_q(n,k,r)$. Lower and upper bounds on
$\cC_q(n,k,r)$ were considered in~\cite{EtVa11a}. Lower bounds are
obtained by analytical methods, and upper bounds are obtained by
constructions of the related $q$-covering designs.

Grassmannian codes and $q$-covering designs are also of interest
in the context of projective geometry. A $k$-spread in PG($n,q$) is
a $q$-covering design $\C_q(n+1,k+1,1)$. Similarly, the values of
$\cC_q(n,n-1,r)$ and $\cC_q(n,n-2,r)$, as well as some related values,
were studied in the context of projective
geometry~\cite{Beu74,Beu79,BeUe91,BoBu66,EiMe97,Lun99,Met03}.
The related structure in projective geometry is a dual structure
to a $q$-covering design and it is called a {\it blocking set}.
A set $\T$ of $t$-subspaces in PG($n,q$) such that every $s$-subspace
is incident with at least one element of $\T$ is called a blocking set.
Such a design is a $q$-analog of the well-known {\it Tur\'{a}n design}~\cite{Caen83,Caen91}. The
dual subspaces of the blocking set form a $q$-covering design $\C_q(n+1,n-t,n-s)$.
Blocking sets were considered for example in~\cite{Met03}.
We note that there is a difference of one in the dimension of
subspaces in the Grassmannian and the dimension of the same subspace the projective geometry.
In other words, $r$-subspaces in PG($n,q$) are $(r+1)$-dimensional subspaces in $\F_q^{n+1}$.

In this paper we consider upper bounds on $\cC_q(n,k,r)$ based
mainly on lifting of maximum rank distance codes combined with
other combinatorial structures. The rest of this paper is
organized as follows. In Section~\ref{sec:bounds} we discuss the
the known bounds on $\cC_q(n,k,r)$ and their implications on the
behavior of the value $\cC_q(n,k,r)$. In
Section~\ref{sec:known_concepts} we introduce the ingredients for
our constructions, lifting of maximum rank distance codes,
subspace transversal designs, and partitions of $\cG_2(4,2)$ into
spreads. In Sections~\ref{sec:bnk2} and~\ref{sec:bnk3} we present
bounds on $\cC_2(n,k,2)$ and $\cC_2(n,k,3)$ and discuss the
density of the obtained $q$-covering designs compared to the
well known covering bound. In Section~\ref{sec:tables} we present
tables of the currently known bounds on $\cC_2(n,k,r)$ for $5 \leq n \leq 10$. In
Section~\ref{sec:conclusion} we present a sequence of problems for
further research. The problems suggest various construction
methods for $q$-covering designs. Before we proceed to the results
of this paper we want to make a small important remark. Although
we will assume throughout that the ambient space is $\F_q^n$, we
point out that our results hold for an arbitrary $n$-dimensional
vector space over $\F_q$.

\section{Known Bounds}
\label{sec:bounds}

In this section we present known bounds on $\cC_q(n,k,r)$.
Most of these bounds will be used later to obtain specific bounds,
mainly because of the recursive nature of the new bounds,
or since the known bounds can be
used as initial conditions for the recursive equations. Even so,
our paper is devoted to new upper bounds on $\cC_q(n,k,r)$, we
will consider also the lower bounds as we can use the lower bounds to
examine how good are the upper bounds. The first
bound is the $q$-analog Sch\"{o}nheim bound~\cite{Sch64} given
in~\cite{EtVa11a}. This is the lower bound which will be frequently used
in our tables.

\begin{theorem}
\label{thm:schonheim} $\cC_q(n,k,r) \geq \left\lceil
\frac{q^n-1}{q^k-1} \cC_q(n-1,k-1,r-1) \right\rceil$.
\end{theorem}

The next theorem can be obtained by iterating
Theorem~\ref{thm:schonheim} or just by noting that each
$r$-dimensional subspace of $\F_q^n$ must be contained in at least
one element of a $q$-covering design $\C_q(n,k,r)$. This bound is
known as the \emph{covering bound}.

\begin{theorem}
\label{thm:covering_bound} $\cC_q(n,k,r) \geq
\begin{small} \frac{\sbinomq{n}{r}}{\sbinomq{k}{r}} \end{small}$
with equality holds if and
only if a Steiner structure $\dS_q(r,k,n)$ exists.
\end{theorem}

Another lower bound given in~\cite{EtVa11a} is a $q$-analog of a
theorem given by de Caen in~\cite{Caen83,Caen91}.

\begin{theorem}
\label{thm:bound_k,k-1} $\cC_q(n,k,k-1) \geq
\frac{(q^k-1)(q-1)}{(q^{n-k}-1)^2} \sbinomq{n}{k+1}$.
\end{theorem}

The next theorem given in~\cite{EtVa11a} is used infinitely many times.
\begin{theorem}
\label{thm:lengthening} $\cC_q (n+ 1 , k + 1 , r) \leq \cC_q (n,k,r)$.
\end{theorem}
Theorem~\ref{thm:lengthening} implies a very interesting property on the behavior
of optimal $q$-design coverings.
\begin{cor}
\label{cor:leng} For any given $r >0$ and $\delta >0$ there exists
a constant $c_{q,\delta,r}$ and an integer $n_0$ such that for
each $n > n_0$, $\cC_q(n,n-\delta,r)=c_{q,\delta,r}$.
\end{cor}
\begin{remark}
The value of $c_{q,\delta,r}$ in Corollary~\ref{cor:leng} can be
derived from the results which follow in this section (see Theorem~\ref{thm:n_spreads}).
\end{remark}

The next two theorems~\cite{EtVa11a} can be viewed as
complementary results. They provide two families for which we know
the exact value of $\cC(n,k,r)$.

\begin{theorem}
\label{thm:optc2} If $1 \leq k \leq n$, then
$\cC_q(n,k,1)=\left\lceil \frac{q^n-1}{q^k-1} \right\rceil$.
\end{theorem}

\begin{theorem}
\label{thm:Turan=1} If $1 \leq r \leq n-1$, then $\cC_q(n,n-1,r) = \frac{q^{r+1}-1}{q-1}$.
\end{theorem}

The next theorem is a consequence of the main recursive
construction. It is given in~\cite{EtVa11a} and it is frequently used
in our tables. As the $q$-covering designs obtained by this
construction might be the building blocks for other constructions
it is presented for completeness and understanding
the other constructions. In the sequel we denote by $\Span{A}$
the subspace of $\F_q^n$ which is spanned by the the set of
elements in $A \subset \F_q^n$.

\begin{const}
\label{con:recursiveC}
Let us represent $\F_q^n$ as
$\{ (x,\alpha) ~:~ x \in \F_q^{n-1},~ \alpha \in \F_q \}$.
Suppose that  $\dS_1$ is a $q$-covering design
$\C_q(n-1,k,r)$ in $\F_q^{n-1}$ and  $\dS_2$ is a
$q$-covering design $\C_q(n-1,k-1,r-1)$ in
$\F_q^{n-1}$. Given a subspace $X$ of $\F_q^{n-1}$, we
define a corresponding subspace $X \times \{0\}$ of $\F_q^n$ as
follows: $ X \times \{0\} = \{ (\vvv,0)
\in \F_q^n ~:~ \vvv \in
X \} $.
Note that if $\dim X = k-1$, then there are
exactly $q^{n-k}$ distinct subspaces of the form
$(X\times \{0\} ) \oplus \Span{\{(x,1)\}}$, each
of dimension $k$ (since we can choose $x$ from any one of the
$q^{n-k}$ cosets of $X$ in $\F_q^{n-1}$). With this, we now
define the sets $\dS'_1$ and $\dS'_2$ as
follows:\vspace{1.00ex}
$$
\dS'_1  \deff\ \bigl\{ X\times \{0\} \subset \F_q^n ~:~ X \in
\dS_1 \bigr\},
$$
$$
\dS'_2 \deff\ \bigl\{ (X \times \{0\} ) \oplus \Span{\{(x,1)\}}
\subset \F_q^n ~:~ X \in \dS_2, ~x \in \F_q^{n-1} \bigr\}.
$$
Let $\dS' = \dS'_1 \cup \dS'_2$.
\end{const}

\begin{theorem}
\label{thm:recursiveC} $\dS'$ is a $q$-covering design $\C_q(n,k,r)$,
and hence $\cC_q (n,k,r) \leq q^{n-k} \cC_q
(n-1,k-1,r-1) + \cC_q (n-1,k,r)$.
\end{theorem}

Normal spreads~\cite{Lun99}, also known as geometric spreads~\cite{BeUe91},
are used to prove the following values of $\cC_q(n,k,r)$~\cite{BlEt11}.

\begin{theorem}
\label{thm:n_spreads}
$\cC_q (vm+\delta,vm-m+\delta,v-1) =
\frac{q^{vm}-1}{q^m-1}$ for all $v \geq 2$, $m \geq 2$, and $\delta \geq 0$.
\end{theorem}


Theorem~\ref{thm:n_spreads} is one of the general results on the minimal
size of a $q$-covering design which was solved in the context of projective
geometry~\cite{BeUe91}. Some of the results stated in this
section were proved before in terms of projective geometry.
In the notation of projective geometry we have that $\cC_q(n,k,r)$
is the smallest size of a set $\T$ which contains $(k-1)$-subspaces
of PG($n-1,q$) such that every $(r-1)$-subspace is contained
in at least one element of $\T$. It is also equal
the smallest size of a set $\T'$ which contains
$(n-k-1)$-subspaces of PG($n-1,q$),
such that every $(n-r-1)$-subspace contains at least one element of $\T'$.

\begin{remark}
The corresponding $q$-covering design $\C_q(n,k,r)$ has $k$-dimensional
subspaces of~$\F_q^n$. The dual $(n-k)$-dimensional subspaces
of $\F_q^n$ are the $(n-k-1)$-subspaces of $\T'$ from the
projective geometry PG($n-1,q$). We note again that there is a difference
of one in the definition of dimension between the Grassmannian and the
projective geometry.
\end{remark}

Theorem~\ref{thm:Turan=1} was proved before in the context of
projective geometry by Bose and Burton in~\cite{BoBu66}.
Another lower bound was given in~\cite{EiMe97} by considering sets
of lines in PG($2s,q$) contained in $s$-subspaces.
\begin{theorem}
\label{thm:EiMe}
$\cC_q (2s+1,2s-1,s) \geq \frac{q^{2s+2}-q^2}{q^2-1} + \frac{q^{s+1}-1}{q-1}$ for
every integer $s \geq 2$.
\end{theorem}

Metsch~\cite{Met03} also gave a construction for a set of lines in
PG($2s+x-1,q$), for every $1 \leq x \leq s$, contained in $s$-subspaces,
which yields the following theorem:
\begin{theorem}
\label{thm:Met03a}
For any given integers $q \geq 2$, $1 \leq x \leq s$ we have
$\cC_q (2s+x,2s+x-2,s+x-1)  \leq \frac{q^{2s+2x}-q^{2x}}{q^2-1}
+ \frac{q^x-1}{q-1} \cdot \frac{q^{s+x}-q^{x-1}}{q-1}$.
\end{theorem}

Finally, also Theorem~\ref{thm:optc2} was proved in terms of
projective geometry. In projective geometry $\cC_q(n+1,k+1,1)$
is the minimal number of $k$-subspaces in PG($n,q$) such that
each point of PG($n,q$) is contained in at least one of these subspaces.
The solution obtained in Theorem~\ref{thm:optc2} was obtained
before by Beutelspacher~\cite{Beu79}.

The proofs for all these results and related results
in projective geometry were also given by Metsch in~\cite{Met03}.

\section{Basic Known Concepts for the Construction}
\label{sec:known_concepts}

In this section we introduce a few concepts which are used in our
constructions: lifting of maximum rank distance (MRD in short)
codes, subspace transversal designs, spreads,
and partition of the Grassmannian
$\cG_2(4,2)$ into disjoint spreads.

\subsection{MRD codes and subspace transversal designs}

For two $k \times \ell$ matrices $A$ and $B$ over $\F_q$ the {\it
rank distance} is defined by
$$
d_R (A,B) \deff \text{rank}(A-B)~.
$$
A  $[k \times \ell,\varrho,\delta]$ {\it rank-metric code} $\cC$
is a linear code, whose codewords are $k \times \ell$ matrices
over $\F_q$; they form a linear subspace with dimension $\varrho$
of $\F_q^{k \times \ell}$, and for each two distinct codewords $A$
and $B$ we have that $d_R (A,B) \geq \delta$ (clearly,
$\delta \leq \min \{k , \ell \}$). For a $[k \times
\ell,\varrho,\delta]$ rank-metric code $\cC$ it was proved
in~\cite{Del78,Gab85,Rot91} that
\begin{equation}
\label{eq:MRD} \varrho \leq \text{min}
\left\{k(\ell-\delta+1),\ell(k-\delta+1) \right\}~.
\end{equation}
This bound is attained for all possible parameters and the codes
which attain it are called {\it maximum rank distance} codes (or
MRD codes in short).

A subset $\C$ of $\cG_q(n,k)$ is called an  $(n,M,d_S,k)_q$
($(n,M,d_S,k)$ if $q=2$) {\it constant
dimension code} if it has size $M$ and minimum distance $d_S$,
where the distance function in $\cG_q(n,k)$ is defined by
\begin{equation*}
\label{def_subspace_distance} d_S (X,\!Y) \deff  2k
-2 \dim\bigl( X\, {\cap}Y\bigr),
\end{equation*}
for any two subspaces $X$ and $Y$ in $\cG_q(n,k)$.

There is a close connection between constant dimension codes and
rank-metric codes~\cite{EtSi09,SKK08}.
Let $A$ be a $k \times \ell$ matrix over $\F_q$ and let $I_k$ be the
$k \times k$ identity matrix. The matrix $[ I_k ~ A ]$ can be
viewed as a generator matrix of a $k$-dimensional subspace of
$\F_q^{k+\ell}$, and it is called the \emph{lifting} of
$A$~\cite{SKK08}.
\begin{example}
Let $A$ and $[I_3 ~A]$ be the following matrices over $\F_2$
$$A=\left( \begin{array}{ccc}
1& 1 & 0\\
0& 1 & 1\\
0& 0 & 1
\end{array}
\right) ~, ~~ [I_3 ~ A] =\left( \begin{array}{cccccc}
1&0&0&1& 1 & 0\\
0&1&0&0& 1 & 1\\
0&0&1&0& 0 & 1
\end{array}
\right),
$$
then the subspace obtained by the lifting of $A$ is given by the
following $8$ vectors:
$$(1 0 0 1 1 0),
(0 1 0 0 1 1), (0 0 1 0 0 1),(1 1 0 1 0 1),$$
$$(1 0 1 1 1 1),(0 1 1 0 1 0),
(1 1 1 1 0 0),(0 0 0 0 0 0).
$$
\end{example}

A constant dimension code $\mathbb{C}$ such that all its codewords
are lifted codewords of an MRD code is called a \emph{lifted MRD
code}~\cite{SKK08}. This code will be denoted by $\CMRD$.

\begin{theorem}\cite{SKK08}
\label{trm:param lifted MRD} If $\cC$ is a $[k \times (n-k),
(n-k)(k-\delta +1),\delta ]$ MRD code then $\C^{\textmd{MRD}}$ is
an $(n,q^{(n-k)(k-\delta+1)}, 2\delta, k)_{q}$ code.
\end{theorem}

\begin{remark}
The parameters of the $[k \times (n-k), (n-k)(k-\delta +1),\delta
]$ MRD code $\cC$ in Theorem~\ref{trm:param lifted MRD} imply
that $k \leq n-k$, by (\ref{eq:MRD}).
\end{remark}

In the sequel we will assume that $q=2$,
even so some of the results can be generalized for
general $q$, where $q$ is a power of a prime number.

A \emph{subspace transversal  design} of groupsize $2^{n-k}$,
block dimension $k$, and \emph{strength}~$t$, denoted by
$\text{STD} (t, k, n-k)$, is a triple
$(\mathbb{V}^{(n,k)},\mathbb{G},\mathbb{B})$, where
$\mathbb{V}^{(n,k)}$ is a set of \emph{points},
$\mathbb{G}$ is a set of groups, and $\mathbb{B}$ is
a set of blocks. These three sets must satisfy the following five
properties:

\begin{enumerate}
\item $\mathbb{V}^{(n,k)}$ is the set of all vectors from $\F_2^n$,
which do not start with $k$ {\it zeroes};
$|\mathbb{V}^{(n,k)}|= (2^k-1) 2^{n-k}$ (the \emph{points});

\item $\mathbb{G}$ is a partition of $\mathbb{V}^{(n,k)}$ into
$2^k-1$ classes of size $2^{n-k}$ (the \emph{groups});

\item $\mathbb{B}$ is a collection of $k$-dimensional
subspaces of $\F_2^n$ which contain nonzero vectors
only from $\mathbb{V}^{(n,k)}$ (the
\emph{blocks});

\item each block meets each group in exactly one point;

\item every $t$-dimensional subspace (with points from $\mathbb{V}^{(n,k)}$) which
meets each group in at most one point is contained in exactly
one block.
\end{enumerate}

Let $\V_x^{(n,\ell)}$, $x \in \F_2^{\ell}$ denote the set of all vectors
from $\F_2^n$ which start with the binary vector~$x$ of length $\ell$.
Clearly, $\V_0^{(n,\ell)}$ is isomorphic to $\F_2^{n-\ell}$,
$\V^{(n,\ell)} \cup V_0^{(n,\ell)}$ is isomorphic to $\F_2^{n}$,
and $\V_0^{(n,\ell)} \cup
\V_x^{(n,\ell)}$ is isomorphic to~$\F_2^{n-\ell+1}$.

\begin{theorem}\cite{EtSi11}\label{thm:STD}
The codewords of an $(n,2^{(n-k)(k-\delta+1)},2\delta,k)$ code
$\C^{\textmd{MRD}}$ form the blocks of a
$\text{STD} (k-\delta+1,k, n-k)$, with the set of points
$\mathbb{V}^{(n,k)}$ and the set of groups $\mathbb{V}^{(n,k)}_x$,
$x\in \F_2^k$.
\end{theorem}

By Theorem~\ref{thm:STD} we have that if the code $\CMRD$ is part of a
$q$-covering design $\C_2(n,k,r)$, $r=k-\delta+1$, then each
$r$-dimensional subspace $X$, for which $\dim (X \cap \V^{(n,k)}_0)=0$, meets each group
of the corresponding subspace transversal design in at most
one point and thus, is
contained in an element of $\CMRD$. Hence, if $\CMRD$ is part of the
final minimal $q$-covering design $\C$ then for each element
$Z \in \C$ not in $\CMRD$ we must have $\dim (Z \cap \V^{(n,k)}_0)>0$.

\subsection{Partition of $\cG_2(4,2)$ into disjoint spreads}
\label{sec:dis_spread}

The description in this subsection will be very specific in its
parameters, although the known results presented are generalized
for other parameters. But, we have restricted ourself only for
those parameters which are needed in the sequel.

A $k$-\emph{spread} $\dS$ in $\cG_q(n,k)$ (or $\F_q^n$) is a set of
$k$-dimensional subspaces of $\F_q^n$, for which each one-dimensional
subspaces of $\cG_q(n,1)$ is contained in exactly one element of
$\dS$. Clearly, $\dS$ is a Steiner structure $\dS_q(1,k,n)$ and it
exists if and only if $k$ divides $n$.

There are $35=\sbinomtwo{4}{2}$ two-dimensional subspaces in
$\F_2^4$. These subspaces can be partitioned into seven 2-spreads
$\dP_0,~\dP_1, \ldots, \dP_6$ each one of size 5~\cite{Beu74,ZZS}.
This is well known and it relates to the well known fifteen
schoolgirls problem~\cite{AbFu96}, to the disjoint translates of
the Preparata code~\cite{ZZS}, and it is a 2-parallelism of
$\cG_2(n,2)$, $n$ even, which is also generalized for $q>2$~\cite{Beu74}. The
properties which follow are also well known.

Each subspace of $\dP_i$, $0 \leq i \leq 6$, partitions $\F_2^4$
into four additive cosets of itself. Consider the set of all
such cosets for $\dP_i$, namely:~
$$
P_i  \deff  \Bigl\{ \{\uuu,\uuu + \vvv_1,\uuu + \vvv_2, \uuu
+ \vvv_3\} ~:~ \{ {\bf 0},\vvv_1,\vvv_2,\vvv_3\} \in \dP_i, ~\uuu
\,{\in}\, \F_2^4 \Bigr\}.
$$
The set $P_i$ will be called a {\it spread translate}. Since the size
of a 2-spread in $\F_2^4$ is 5, and each 2-dimensional subspaces of $\F_2^4$
has 4 cosets it follows that $P_i$ consists of 20 distinct 4-subsets for each $i$,
$0 \leq i \leq 6$. These 20 subsets are partitioned into
5 parallel classes of size 4. Each parallel class contains the cosets of a different
2-dimensional subspace of $\dP_i$.

\begin{lemma}
\label{lem:Ppairs}
For any given $i$, $0 \leq i \leq 6$, each pair
$\{ \uuu,\vvv \} \subset \F_2^4$, $\uuu \neq \vvv$, appears in exactly one element of $P_i$.
\end{lemma}

\begin{lemma}
\label{lem:Psubspace}
For any given $i$, $0 \leq i \leq 6$ and an element $\{ a_1,a_2,a_3,a_4 \} \in P_i$, the subspace
$\Span{\{ (z,a_1),(z,a_2),(z,a_3),(z,a_4) \}}$, $z \in \F_2^\ell$, $\ell \in \N$,
is a 3-dimensional subspace
defined on $\V_0^{(\ell+4,\ell)} \cup \V_z^{(\ell+4,\ell)}$.
\end{lemma}

\begin{lemma}
\label{lem:P_V0} The set $$\bigl\{ \Span{\{
(z,a_1),(z,a_2),(z,a_3),(z,a_4) \}} \cap \V_0^{(\ell+4,4)} ~:~ \{
a_1,a_2,a_3,a_4 \} \in P_i ,~ 0 \leq i \leq 6 \bigr\},~z \in \F_2^\ell,~\ell \in \N~,$$ contains
all the 2-dimensional subspaces of $\V_0^{(\ell+4,\ell)}$ (which is isomorphic to $\F_2^4$).
\end{lemma}

\section{On the Value of $\cC_2(n,k,2)$}
\label{sec:bnk2}

In this section we will consider lower and upper bounds on
$\cC_2(n,k,2)$. By Theorem~\ref{thm:n_spreads}, $\cC_q
(3m+\delta,2m+\delta,2) = \frac{q^{3m}-1}{q^m-1}$ for all $m \geq
2$ and $\delta \geq 0$. This implies a sequence of exact values of
$\cC_2(n,k,2)$. For values of $k$ not covered by this theorem we
will present one general construction and one specific one which
will be also used as a basis for another recursive construction.
We will also use the recursions implied by
Theorems~\ref{thm:lengthening} and~\ref{thm:recursiveC}.

In the sequel we will represent nonzero elements of the finite field
$\F_{2^n}$ in two different ways. The first one is by $n$-tuples over $\F_2$
(in other words, $\F_{2^n}$ is represented by $\F_2^n$)
and the second one is by powers of a primitive element $\alpha$ in $\F_{2^n}$.
We will not distinguish between these two isomorphic representations. When $n$-tuple $z$
over $\F_2$ will be multiplied by an element $\beta \in \F_{2^n}$ we
will view $z$ as an element in $\F_{2^n}$ and the result will be
an element in $\F_{2^n}$ which is also represented by an $n$-tuple
over $\F_2$ (an element in $\F_2^n$). Also, when we write
$\V_\gamma^{(n,\ell)}$, where $\gamma \in \F_{2^\ell}$, it is
the same as writing $\V_x^{(n,\ell)}$, $x \in \F_2^\ell$, where $x$
is the binary $\ell$-tuple which represents $\gamma$.

For a set $S \subseteq \F_2^n$ and a nonzero element $\beta\in \F_{2^n}$,
we define $\beta S \deff \{ \beta x ~:~ x \in S \}$. We
note that we can take the set $S$ to be a subspace. The following lemma is a
simple observation.
\begin{lemma}
\label{lem:multiplier}
If $X$ is a $k$-dimensional subspace of $\F_2^n$ and $\beta$ is a
nonzero element of $\F_2$ then $\beta X$ is also a
$k$-dimensional subspace of $\F_2^n$.
\end{lemma}
Another important and simple observation is the following result.
\begin{lemma}
\label{lem:dis_subspaces}
Let $\alpha$ be a primitive element in $\F_{2^k}$, $X$ be an $r$-dimensional subspace
of $\F_2^k$. If $2^k-1$ and $2^r-1$ are relatively primes then the set
$\{ \alpha^j X ~:~ 0 \leq j \leq 2^k-2 \}$ contains $2^k-1$ distinct $r$-dimensional subspaces.
\end{lemma}

For $k \geq 3$ we define the following 6 sets of size $2^{k-1}$
in $\F_2^k$. $B_1$ contains all the elements of $\F_2^k$ which
start with a {\it zero}; $B_2$ contains all the elements of
$\F_2^k$ which start with an {\it one}; $B_3$ contains all the
elements of $\F_2^k$ which start with 00 or 10; $B_4$ contains
all the elements of $\F_2^k$ which start with 01 or 11; $B_5$
contains all the elements of $\F_2^k$ which start with 00 or~11;
$B_6$ contains all the elements of $\F_2^k$ which start with 01
or 10. The following two lemmas are readily verified.

\begin{lemma}
\label{lem:pairsB} If $\{ x,y \}$ is a pair of elements from
$\F_2^k$ then there exists at least one $i$, $1 \leq i \leq 6$,
such that $\{ x,y \} \subset B_i$.
\end{lemma}
\begin{lemma}
\label{lem:k_subspaceB} For any given $i$, $1 \leq i \leq 6$, and
$z \in \F_2^k \setminus \{ {\bf 0} \}$, the subspace
$\Span{\{ (z,x) ~:~ x \in B_i
\}}$, is a $k$-dimensional subspace of $\F_2^{2k}$.
\end{lemma}
Let $\alpha$ be a primitive element in $\F_{2^k}$ and assume that
the elements of $B_i$, $1 \leq i \leq 6$ are viewed as elements of
$\F_{2^k}$. For each $j$, $0 \leq j \leq 2^k-2$ we define the
following set with six subsets of $\V_{\alpha^j}^{(2k,k)}$ (the elements are
from $\F_2^{2k}$),
$$
S_j \deff \bigl\{ \{ (\alpha^j , \alpha^j x ) ~:~ x \in B_i \} ~:~
1 \leq i \leq 6 \bigr\}
$$
An immediate consequence of Lemmas~\ref{lem:multiplier},~\ref{lem:pairsB}, and~\ref{lem:k_subspaceB} is
the following lemma.
\begin{lemma}
\label{lem:2subsB} The set $\dS_j \deff \{ \Span{A} ~:~ A \in S_j
\}$ contains six $k$-dimensional subspaces of $\V_0^{(2k,k)} \cup \V_{\alpha^j}^{(2k,k)}$.
Each 2-dimensional subspace $X$ of $\V_0^{(2k,k)} \cup
\V_{\alpha^j}^{(2k,k)}$, for which $X$ is not contained in $\V_0^{(2k,k)}$,
is a subspace of at least one $k$-dimensional subspace of $\dS_j$.
\end{lemma}

\begin{theorem}
\label{thm:simpleCMRD}
$\cC_2(2k,k,2) \leq 2^{2k}+6 \cdot (2^k-1)$.
\end{theorem}
\begin{proof}
Let $\CMRD$ be a $(2k,2^{2k},2(k-1),k)$ lifted MRD code, and let $\T$ be the
set of blocks in the
corresponding STD$(2,k,k)$. Define
$$
\C \deff \T \cup \bigcup_{j=0}^{2^k-2} \dS_j ~.
$$
We claim that $\C$, which contains $2^{2k}+6 \cdot (2^k-1)$
subspaces, is a $q$-covering design $\C_2 (2k,k,2)$.

By Theorem~\ref{thm:STD}, each 2-dimensional subspace $X$ of
$\F_2^{2k}$, for which $\dim ( X \cap \V_0^{(2k,k)} )=0$, is contained in
an element of $\T$. By Lemma~\ref{lem:2subsB}, each 2-dimensional subspace
$X$ of $\F_2^{2k}$, for which $\dim ( X \cap \V_0^{(2k,k)} ){=1}$, is
contained in an element of $\bigcup_{j=0}^{2^k-2} \dS_j$. Therefore, we only
have to prove that for each 2-dimensional subspace $X$ of $\V_0^{(2k,k)}$,
there exists a $k$-dimensional subspace $Y$ of $\C$ such that $X
\subset Y$.

Let $Z$ be a $k$-dimensional subspace such that $Z \in \C$ and $Z
\subset \V_0^{(2k,k)} \cup \V_1^{(2k,k)}$, where $1=\alpha^0$
(thus $Z \in \dS_0)$. $Z$ can be written as $Z_0 \cup Z_1$, where
$Z_0 \subset \V_0^{(2k,k)}$ and $Z_1 \subset \V_1^{(2k,k)}$. For each $j$, $0 \leq j
\leq 2^k-2$, $\alpha^j Z$ is also a $k$-dimensional subspace of
$\C$, and $\alpha^j Z \subset \V_0^{(2k,k)} \cup \V_{\alpha^j}^{(2k,k)}$.
Thus, $\alpha^j Z$ can be written as $Z_{0,j} \cup Z_{\alpha^j}$
($Z_{0,0}=Z_0$), where $Z_{0,j} \subset \V_0^{(2k,k)}$
and $Z_{\alpha^j} \subset \V_{\alpha^j}^{(2k,k)}$. Clearly,
$Z_0$ is a $(k-1)$-dimensional subspace of $\V_0^{(2k,k)}$. By
Lemma~\ref{lem:dis_subspaces} the set $\Z
\deff \{ Z_{0,j} ~:~ 0 \leq j \leq 2^k-2 \}$ contains $2^k-1$
distinct $(k-1)$-dimensional subspaces of $\V_0^{(2k,k)}$.
Since $\V_0^{(2k,k)}$ is a
$k$-dimensional subspace it contains $2^k-1$ distinct
$(k-1)$-dimensional subspaces. Hence, $\Z$ contains all the
$(k-1)$-dimensional subspaces of $\V_0^{(2k,k)}$. Each 2-dimensional
subspace of $\V_0^{(2k,k)}$ is a subspace of some $(k-1)$-dimensional
subspace of $\V_0^{(2k,k)}$ and since $\alpha^j Z \in \C$ the proof is
completed.
\end{proof}

\begin{theorem}
\label{thm:b396} $\cC_2(7,3,2) \leq 396$.
\end{theorem}
\begin{proof}
Let $\CMRD$ be a $(7,256,4,3)$ lifted MRD code, and let $\T$ be
the set of blocks in the corresponding STD$(2,3,4)$. Let $\alpha$ be a primitive
element in $\F_8$. Recall the definition of $P_i$, $0 \leq i \leq
6$, given in subsection~\ref{sec:dis_spread}. For each $0 \leq i
\leq 6$, we construct the following set of 3-dimensional subspaces
$$\B_{\alpha^i} \deff \bigl\{ \Span{\{ (\alpha^i,a_1),(\alpha^i,a_2),(\alpha^i,a_3),(\alpha^i,a_4) \}} ~:~
\{ a_1,a_2,a_3,a_4 \} \in P_i  \bigr\}$$
which are contained in $\V_0^{(7,3)} \cup
\V_{\alpha^i}^{(7,3)}$, where the elements of $\{
(\alpha^i,a_1),(\alpha^i,a_2),(\alpha^i,a_3),(\alpha^i,a_4) \}$,
${\{ a_1,a_2,a_3,a_4 \} \in P_i}$, are embedded on $\V_{\alpha^i}^{(7,3)}$.
Let
$$
\C \deff \T \cup \bigcup_{0 \leq i \leq 6} \B_{\alpha^i} ~.
$$
We claim that $\C$ is a $q$-covering design $\C_2 (7,3,2)$.

By Theorem~\ref{thm:STD}, each 2-dimensional subspace $X$ of
$\F_2^7$, for which $\dim ( X \cap \V_0^{(7,3)} )=0$, is contained in
an element of $\T$. By Lemmas~\ref{lem:Ppairs} and~\ref{lem:Psubspace}, each
2-dimensional subspace $X$ of $\F_2^7$, for which $\dim ( X \cap
\V_0^{(7,3)} )=1$ and $X \subset \V_0^{(7,3)} \cup \V_{\alpha^i}^{(7,3)}$, is contained in
at least one subspace of~$\B_{\alpha^i}$ and hence it is covered
by $\C$. By Lemma~\ref{lem:P_V0}, each 2-dimensional subspace $X$
of $\F_2^7$, for which $\dim ( X \cap \V_0^{(7,3)} )=2$ is contained in at
least one subspace of $\bigcup_{0 \leq i \leq 6} \B_{\alpha^i}$
and hence it is covered by $\C$.

Thus, $\C$ is a $q$-covering design
$\C_2(7,3,2)$. $\C$ contains
$256+7 \cdot 20=396$ subspaces and hence $\cC_2(7,3,2) \leq
396$.
\end{proof}

The constructions in the proofs of Theorems~\ref{thm:simpleCMRD}
and~\ref{thm:b396}, and other $q$-covering designs $\C_2(n,k,2)$, can
be used in a recursive construction implied by the following
theorem.
\begin{theorem}
\label{thm:improved_CMRD} Let $n \geq 2k$ and let $\dS$ be a
$q$-covering design $\C_2(n-k+1,k,2)$ in which there exists an
$(n-k)$-dimensional subspace $U \subset \F_2^{n-k+1}$ and a set
$\dS_1= \{ X ~:~ X \in \dS,~ X \subset U  \}$, $|\dS_1|=c$. Then
$\cC_2(n,k,2) \leq 2^{2(n-k)} +(2^k-1)|\dS| - (2^k-2)c$.
\end{theorem}
\begin{proof}
Let $\CMRD$ be an $(n,2^{2(n-k)},2(k-1),k)$ lifted MRD code, and
let $\T$ be the set of blocks in the corresponding STD$(2,k,n-k)$.

Let $\C_1$ consist of the subspaces of $\dS \setminus \dS_1$
contained in $\V_0^{(n,k)} \cup \V_x^{(n,k)}$, for each $x \in \F_2^k \setminus
\{ {\bf 0} \}$, where $\V_0^{(n,k)}$ coincides with $U$.

Let $\C_2$ consist of the subspaces of $\dS_1$ on the points of
$\V_0^{(n,k)}$.

We define
$$
\C \deff \T \cup \C_1 \cup \C_2 ~.
$$
It can be easily verified that $\C$ is a $\C_2(n,k,2)$ covering
design of size $2^{2(n-k)} +(2^k-1)|\dS| - (2^k-2)c$, and the
theorem follows.
\end{proof}

\subsection{On the density of $q$-covering designs}

The density of a $q$-covering design $\C_q(n,k,r)$, $\C$, is
defined as the ratio between $| \C |$ and the covering bound
$\left\lceil \begin{tiny} \frac{\sbinomq{n}{r}}{\sbinomq{k}{r}}
\end{tiny} \right\rceil$ as $n$ tends to infinity. Constructions
for which this ratio is equal~1 were considered in~\cite{BlEt11}.
We will consider now this ratio for two cases, $q$-covering designs $\C_2(2k,k,2)$
and $q$-covering designs $\C_2(2n+1,3,2)$.

By Theorem~\ref{thm:simpleCMRD} we have
$\cC_2(2k,k,2) \leq 2^{2k}+6 \cdot (2^k-1)$. The covering bound
in this case is equal to
$$
\left\lceil  \begin{tiny} \frac{\sbinomtwo{2k}{2}}{\sbinomtwo{k}{2}} \end{tiny} \right\rceil
=\left\lceil \frac{(2^{2k}-1)(2^{2k-1}-1)}{(2^k-1)(2^{k-1}-1)} \right\rceil =
\left\lceil \frac{2^{3k-1}+2^{2k-1}-2^k-1}{2^{k-1}-1} \right\rceil =2^{2k}+3 \cdot 2^k+5 ~.
$$
The ratio between the size of the $q$-covering design and the covering bound is
$\frac{2^{2k}+6 \cdot 2^k -6}{2^{2k}+3 \cdot 2^k+5}$ which approaches 1,
when $k$ tends to infinity. Thus, the construction in the proof
of Theorem~\ref{thm:simpleCMRD} yields a $q$-covering design which is asymptotically
optimal.

Now, we consider the value of $\cC_2(2n+1,3,2)$. The covering bound
in this case is
\begin{equation}
\label{eq:lower_3_2}
\left\lceil
\begin{tiny} \frac{\sbinomtwo{2n+1}{2}}{\sbinomtwo{3}{2}} \end{tiny}
\right\rceil =\left\lceil \frac{(2^{2n+1}-1)(2^{2n}-1)}{21} \right\rceil ~.
\end{equation}
As for the upper bound on $\cC_2(2n+1,3,2)$ we use
Theorem~\ref{thm:improved_CMRD} iteratively with the initial
condition $\cC_2(7,3,2) \leq 396$ (see Theorem~\ref{thm:b396}).
Without going into all the specific details of the proofs we can verify the
following properties concerning this upper bound.
\begin{lemma}
\label{lem:rec_pairs} Let Theorem~\ref{thm:improved_CMRD} be
applied iteratively to obtain a $q$-covering design ${\C_2(2n+1,3,2)}$,
$\dS$, starting with the $q$-covering design $\C_2(7,3,2)$ of size 396,
obtained in the proof of Theorem~\ref{thm:b396}. Then each
2-dimensional subspace $Y$ of $\V_0^{(2n+1,3)} \cup \V_x^{(2n+1,3)}$, $x \in \F_2^3
\setminus \{ {\bf 0} \}$, such that $\dim (Y \cap \V_0^{(2n+1,3)})=1$, is
contained in exactly one 3-dimensional subspaces of $\dS$.
\end{lemma}
\begin{proof}
The proof is by induction on $n$. The basis, $n=3$, is the $q$-covering design
$\C_2(7,3,2)$ of size 396,
obtained in the proof of Theorem~\ref{thm:b396}, and the
claim is immediate by Lemma~\ref{lem:Ppairs}. In the induction step
we note that $\V_0^{(2n+1,3)}$ is isomorphic to a union of $\V_0^{(2n-1,3)}$,
$\V_x^{(2n-1,3)}$, $\V_y^{(2n-1,3)}$, and $\V_z^{(2n-1,3)}$, where
$\{ {\bf 0},x,y,z \}$ is a 2-dimensional subspace of $\F_2^3$. For each
$u \in \F_2^3 \setminus \{ {\bf 0} \}$ we have that $\V_u^{(2n+1,3)}$ is isomorphic
to $\bigcup_{v \in \F_2^3 \setminus \{ {\bf 0},x,y,z \}} \V_v^{(2n-1,3)}$.
Now, the claim follows from the induction hypothesis and the fifth property
in the definition of subspace transversal design.
\end{proof}
\begin{cor}
\label{cor:rec_pairs} Let Theorem~\ref{thm:improved_CMRD} be
applied iteratively to obtain a $q$-covering design ${\C_2(2n+1,3,2)}$, $\dS$, starting
with the $q$-covering design $\C_2(7,3,2)$ of size 396,
obtained in the proof of Theorem~\ref{thm:b396}. Then the number
of 3-dimensional subspaces of $\dS$ which are contained in $\V_0^{(2n+1,3)}
\cup \V_x^{(2n+1,3)}$, $x \in \F_2^3 \setminus \{ {\bf 0} \}$, and are not
contained in $\V_0^{(2n+1,3)}$, is $\frac{2^{2n-2}(2^{2n-2}-1)}{3 \cdot 4}$.
\end{cor}
\begin{proof}
Again, the proof of this result is by induction on $n$. The basis is $n=3$
and in the construction given in the proof of Theorem~\ref{thm:b396} the
number of such 3-dimensional subspaces is 20. Assume now that
the number of 3-dimensional subspaces of $\dS'$
(of a $q$-covering design $\C_2(2n-1,3,2)$) which are contained in $\V_0^{(2n-1,3)}
\cup \V_x^{(2n-1,3)}$, $x \in \F_2^3 \setminus \{ {\bf 0} \}$, and are not
contained in $\V_0^{(2n-1,3)}$, is $\frac{2^{2n-4}(2^{2n-4}-1)}{3 \cdot 4}$.
In the induction step the 3-dimensional subspaces of $\dS$ which are contained in $\V_0^{(2n+1,3)}
\cup \V_x^{(2n+1,3)}$, $x \in \F_2^3 \setminus \{ {\bf 0} \}$, and are not
contained in $\V_0^{(2n+1,3)}$ consist first of all the $2^{2(2n-4)}$ subspaces
of the STD$(2,3,2n-4)$ obtained from the $(2n-1,2^{2(2n-4)},4,3)$ lifted MRD code.
These subspaces are joined by the
3-dimensional subspaces of $\dS'$ which are contained in $\V_0^{(2n-1,3)}
\cup \V_u^{(2n-1,3)}$ and are not
contained in $\V_0^{(2n-1,3)}$, for each $u \in \F_2^3 \setminus \{ {\bf 0},x,y,z \}$,
where $\{ {\bf 0},x,y,z \}$ is any subspace of $\F_2^3$.
By the induction assumption the number of these subspaces is
$4 \cdot \frac{2^{2n-4}(2^{2n-4}-1)}{3 \cdot 4}$. Thus, the total number of these
subspaces is $2^{2(2n-4)} + 4 \cdot \frac{2^{2n-4}(2^{2n-4}-1)}{3 \cdot 4} =\frac{2^{2n-2}(2^{2n-2}-1)}{3 \cdot 4}$.
\end{proof}
\begin{lemma}
\label{lem:rec_V0} Let Theorem~\ref{thm:improved_CMRD} be
applied iteratively to obtain a $q$-covering design ${\C_2(2n+1,3,2)}$,
$\dS$, starting with the $q$-covering design $\C_2(7,3,2)$ of size 396,
obtained in the proof of Theorem~\ref{thm:b396}. Then the number
of 3-dimensional subspaces of $\dS$ which are contained in $\V_0^{(2n+1,3)}$ is $3
\sum_{i=0}^{n-4} \frac{2^{2i+4} (2^{2i+4}-1)}{3 \cdot 4}$.
\end{lemma}
\begin{proof}
It follows from the fact that in each iteration of
Theorem~\ref{thm:improved_CMRD}, the 3-dimensional subspaces
contained in $\V_0^{(2n+1,3)}$ of $\F_2^{2n+1}$, from the $q$-covering design
$\C_2(2n+1,3,2)$, are exactly those 3-dimensional subspaces
contained in $\V_0^{(2n-1,3)} \cup \V_x^{(2n-1,3)}
\cup \V_y^{(2n-1,3)} \cup \V_z^{(2n-1,3)}$ of $\F_2^{2n-1}$
(where $\{ {\bf 0},x,y,z \}$ is a 2-dimensional subspace of
$\F_2^3$), from the $q$-covering design
$\C_2(2n-1,3,2)$ used by the recursion. Now, the number of
3-dimensional subspaces of $\dS$ contained in $\V_0^{(2n+1,3)}$ follows from
Corollary~\ref{cor:rec_pairs} and the fact that no element of $\C_2(7,3,2)$
is contained in $\V_0^{(2n+1,3)}$.
\end{proof}
\begin{cor}
\label{cor:Cn_3_2}
$$
\cC_2(2n+1,3,2) \leq 2^{4n-4} + 7 \cdot \frac{2^{2n-2}(2^{2n-2}-1)}{12} + \sum_{i=0}^{n-4} \frac{2^{2i+4} (2^{2i+4}-1)}{4}
$$
\end{cor}
Now, when $n$ tends to infinity
the ratio between the upper bound on $\cC_2(2n+1,3,2)$ given in
Corollary~\ref{cor:Cn_3_2} and the covering bound given in
(\ref{eq:lower_3_2}) is
$$
\frac{2^{4n-4} + \frac{7}{3} 2^{4n-6} + \frac{2^{4n-6}}{15}}{\frac{2^{4n+1}}{21}} = 1.05~.
$$

\section{On the Value of $\cC_2(n,k,3)$}
\label{sec:bnk3}

In this section we will present a few new upper bounds on
$\cC_2(n,k,3)$. By Theorem~\ref{thm:n_spreads}, $\cC_q
(4m+\delta,3m+\delta,3) = \frac{q^{4m}-1}{q^m-1}$ for all $m \geq
2$ and $\delta \geq 0$. This implies a sequence of exact values of
$\cC_2(n,k,3)$. For values of $k$ not covered by this theorem we
will present two specific constructions which will be also used as
a basis for another recursive construction. We will also use the
recursion implied by Theorem~\ref{thm:recursiveC}.

\begin{theorem}
\label{thm:b6897} $\cC_2(8,4,3) \leq 6897$.
\end{theorem}
\begin{proof}
Let $\CMRD$ be a $(8,4096,4,4)$ lifted MRD code, and let $\T$ be
the set of blocks in the corresponding STD$(3,4,4)$.

Recall the definitions of $\dP_i$ and $P_i$, $0 \leq i \leq 6$,
given in subsection~\ref{sec:dis_spread}. Consider one~$\dP_i$
and its spread translate $P_i$. We consider a partition for the
20 elements of $P_i$ into its 5 parallel classes denoted by
$\tilde{P}_{i,j}$, $1 \leq j \leq 5$. For each $\{ {\bf 0},x,y,z \} \in \dP_i$
we form a set of
80 distinct
4-dimensional subspaces of $\F_2^8$:
$$
\bigl\{ \Span{ \{ (x,a_1),(x,a_2),(x,a_3),(x,a_4),(y,b_1) \} } ~:~
\{a_1,a_2,a_3,a_4\},\{b_1,b_2,b_3,b_4\} \in \tilde{P}_{i,j} ,~ 1 \leq
j \leq 5 \bigr\}
$$
Let~$\C_1$ be the union of these $5 \cdot 7=35$ sets. $\C_1$ has
$35 \cdot 80 =2800$ subspaces formed in this way. Let
$$
\C \deff \T \cup \C_1 \cup \{ \V_0^{(8,4)} \} ~.
$$
We claim that $\C$, which contains $4096+2800+1=6897$ subspaces,
is a $q$-covering design $\C_2 (8,4,3)$.

To complete the proof we have to show that each 3-dimensional
subspace of $\F_2^8$ is contained in at least one 4-dimensional
subspace of $\C$. By Theorem~\ref{thm:STD}, each 3-dimensional
subspace $X$ of $\F_2^8$ for which $\dim ( X \cap \V_0^{(8,4)} )=0$ is
contained in a subspace of $\T$.

A 3-dimensional subspace $X$ of $\F_2^8$ for which $\dim ( X \cap
\V_0^{(8,4)} )=1$ has the form
$$\bigl\{ {\bf 0} , (0,u),
(x,a_1),(x,a_2),(y,b_1),(y,b_2), (z,d_1),(z,d_2) \bigr\}~,
$$
where $(0,u) \in \V_0^{(8,4)}$, $u,x,y,z \in \F_2^4 \setminus \{ {\bf 0 } \}$, $x+y+z = {\bf 0}$,
$a_1+a_2=b_1+b_2=d_1+d_2=u$, and $d_1 = a_1 + b_1$. The
2-dimensional subspace $\{ {\bf 0} ,x,y,z \}$ of $\F_2^4$ is
contained in a unique 2-spread $\dP_i$, for some~$i$, $0 \leq i
\leq 6$. By Lemma~\ref{lem:Ppairs}, each one of the pairs $\{ a_1
,a_2 \}$, $\{ b_1 , b_2 \}$, $\{ d_1 , d_2 \}$ appears in exactly
one 4-subset $A$, $B$, $D$, respectively, of $P_i$. Since
$a_1+a_2=b_1+b_2=c_1+c_2=u$ it follows that these three 4-subsets
are cosets of the unique 2-dimensional subspace $\{ {\bf 0},u,v,w
\} \in \dP_i$. Since $d_1=a_1+b_1$ it follows that $\{ {\bf 0 },
(0,u),(0,v),(0,w) \} \cup \{ (x,a) ~:~ a \in A \} \cup \{ (y,b)
~:~ b \in B \} \cup \{ (z,d) ~:~ d \in D \}$ is a 4-dimensional
subspace of $\C_1$, which contains~$X$.

A 3-dimensional subspace $X$ of $\F_2^8$ for which $\dim ( X \cap
\V_0^{(8,4)} )=2$ has the form
$$\bigl\{ {\bf 0} , (0,u_1),(0,u_2),(0,u_3),
(x,a_1),(x,a_2),(x,a_3),(x,a_4) \bigr\}~,
$$
where $(0,u_1),(0,u_2),(0,u_3) \in \V_0^{(8,4)}$ and
$(x,a_1),(x,a_2),(x,a_3),(x,a_4) \in \V_x^{(8,4)}$. By the construction of
$\C_1$ each 2-spread $\dP_i$, $0 \leq i \leq 6$ is used in the
construction and hence by Lemma~\ref{lem:P_V0} the 2-dimensional
subspace $\{ {\bf 0},(0,u_1),(0,u_2),(0,u_3) \}$ is a subspace of
some elements from $\C_1$. There are 80 subspaces in $\C_1$ which
contain 3-dimensional subspaces contained in $\V_0^{(8,4)} \cup \V_x^{(8,4)}$ from
which four contain $X$.

Finally, each 3-dimensional subspace of $\F_2^8$ for which $\dim (
X \cap \V_0^{(8,4)} )=3$ is contained in $\V_0^{(8,4)} \in \C$.

Thus, $\C$ is a $q$-covering design $\C_2(8,4,3)$ with 6897 subspaces.
\end{proof}

%

Theorem~\ref{thm:improved_CMRD} can be modified to obtain bounds
on $\cC_2(n,k,3)$.
\begin{theorem}
\label{thm:improved_CMRDk_3} Let $n \geq 2k$ and let $\dS$ be a
$q$-covering design $\C_2(n-k+2,k,3)$ in which there exist an
$(n-k)$-dimensional subspaces $U_0 \subset \F_2^{n-k+2}$, and
three $(n-k+1)$-dimensional subspace $U_i \subset \F_2^{n-k+2}$,
$i=1,2,3$ such that $U_i \cap U_j = U_0$, $1 \leq i < j \leq 3$.
Assume further that
$\dS_i= \{ X ~:~ X \in \dS,~ X \subset U_i \}$,
$|\dS_i|=c_i$, $i=0,1,2,3$, and $c_1 \leq c_2 \leq c_3$. Then,
$\cC_2(n,k,3) \leq 2^{3(n-k)}
+\sbinomtwo{k}{2}(|\dS|-c_1-c_2-c_3+2c_0) +(2^k-1)(c_1-c_0) +\cC_2
(n-k,k,3)$.
\end{theorem}
\begin{proof}
Let $\CMRD$ be an $(n,2^{3(n-k)},2(k-2),k)$ lifted MRD code, and
let $\T$ be the set of blocks in the corresponding STD$(3,k,n-k)$.

Let $\C_1$ consist of the subspaces of $\dS \setminus (\dS_1 \cup
\dS_2 \cup \dS_3)$ on the points of $\V_0^{(n,k)} \cup \V_x^{(n,k)} \cup \V_y^{(n,k)} \cup
\V_z^{(n,k)}$, for each 2-dimensional subspace $\{ {\bf 0} ,x,y,z \}$ of
$\F_2^k$, where $\V_0^{(n,k)}$ coincides with $U_0$, $\V_x^{(n,k)}$ with $U_1$,
$\V_y^{(n,k)}$ with $U_2$, and $\V_z^{(n,k)}$ with $U_3$. The choice, in which
$\V_x^{(n,k)}$, $\V_y^{(n,k)}$, and $\V_z^{(n,k)}$ are matched with $U_1$, $U_2$,
and $U_3$, is made in a
way that for each $x \in \F_2^k \setminus \{ {\bf 0} \}$, $\V_x^{(n,k)}$
coincides at least once with $U_1$.

Let $\C_2$ consist of the subspaces of $\dS_1 \setminus \dS_0$ on
the points of $\V_0^{(n,k)} \cup \V_x^{(n,k)}$, for each $x \in \F_2^k \setminus
\{ {\bf 0} \}$, where $V_x^{(n,k)}$ coincides with $U_1$.

Let $\C_3$ consist of a $q$-covering design $\C_2(n-k,k,3)$ on the points of $\V_0^{(n,k)}$.

We define
$$
\C \deff \T \cup \C_1 \cup \C_2 \cup \C_3 ~.
$$
It can be easily verified that $\C$ is a $q$-covering design $\C_2(n,k,3)$
with $2^{3(n-k)} +\sbinomtwo{k}{2}(|\dS|-c_1-c_2-c_3+2c_0)
+(2^k-1)(c_1-c_0) +\cC_2(n-k,k,3)$ distinct $k$-dimensional subspaces
and the theorem follows.
\end{proof}
\begin{remark}
Theorem~\ref{thm:improved_CMRDk_3} can be slightly improved for
specific cases by omitting for example the subspaces of $\C_3$, or
taking much less subspaces than required by $\cC_2(n-k,k,3)$. An example
will be given in Section~\ref{sec:tables}. The messy general proof
is omitted as it is of less interest
\end{remark}

\vspace{0.75cm}

\section{Tables}
\label{sec:tables}

This section is devoted for tables with the known lower and upper
bounds on $\cC_2(n,k,r)$ for $1 \leq r \leq k <n$, $5 \leq n \leq
10$. At the end of the section we will demonstrate how the upper bound
on $\cC_2(10,5,3)$ was obtained.

\begin{center}
\hspace{4.00ex}
\begin{tabular}{|c||c@{\hspace{0.50ex}}%
                    c@{\hspace{0.50ex}}%
                    c@{\hspace{0.50ex}}%
                    c|}
\multicolumn{5}{c}{\sf \hspace*{-10ex}Bounds on
$\cC_2(5,k,r)$\hspace*{-10ex}}
\vspace{0.5cm}
\\[0.75ex]
\hline \raisebox{-1.25ex}{\large$r$} &
\multicolumn{4}{c|}{\large$k$\Strut{2.50ex}{0ex}}
\\[-1.00ex]
& 4 & 3 & 2 & 1
\\
\hline\hline $4$ & $^q31^q$\Strut{2.50ex}{0ex} &
\multicolumn{3}{|c|}{}
\\
\cline{3-3} $3$ & $^q15^q$ & $^a155^a$\Strut{2.25ex}{0ex} &
\multicolumn{2}{|c|}{}
\\
\cline{4-4} $2$ & $^q7^q$ & $^e27^m$ & $^a155^a$\Strut{2.25ex}{0ex}
&\multicolumn{1}{|c|}{}
\\
\cline{5-5} $1$ & $^p3^p$ & $^p5^p$ & $^p11^p$ &
$^p31^p$\Strut{2.25ex}{0ex}
\\
\hline
\end{tabular}
\hspace{3.00ex}
\begin{tabular}{|c||c@{\hspace{1.00ex}}%
                    c@{\hspace{0.50ex}}%
                    c@{\hspace{0.50ex}}%
                    c@{\hspace{0.50ex}}%
                    c|}
\multicolumn{6}{c}{\sf Bounds on $\cC_2(6,k,r)$}
\vspace{0.5cm}
\\[0.75ex]
\hline \raisebox{-1.25ex}{\large$r$} &
\multicolumn{5}{c|}{\large$k$\hspace*{3.00ex}\Strut{2.50ex}{0ex}}
\\[-1.00ex]
& 5 &  4 & 3 & 2 & 1
\\
\hline\hline $5$ & $^q63^q$\Strut{2.50ex}{0ex} &
\multicolumn{4}{|c|}{}
\\
\cline{3-3} $4$ & $^q31^q$ & $^a651^a$\Strut{2.25ex}{0ex} &
\multicolumn{3}{|c|}{}
\\
\cline{4-4} $3$ & $^q15^q$ & $^s114{-}122^m$ &
$^a1395^a$\Strut{2.25ex}{0ex} &\multicolumn{2}{|c|}{}
\\
\cline{5-5} $2$ & $^q7^q$ & $^s21^n$ & $^s99{-}106^c$ &
$^a651^a$\Strut{2.25ex}{0ex} & \multicolumn{1}{|c|}{}
\\
\cline{6-6} $1$ & $^p3^p$ & $^p5^p$ & $^p9^p$ & $^p21^p$ &
$^p63^p$\Strut{2.25ex}{0ex}
\\
\hline
\end{tabular}
\end{center}

\vspace{1.0cm}

\begin{center}
{\sf Bounds on $\cC_2(7,k,r)$}
\vspace{0.5cm}
\\[1.00ex]
\begin{tabular}{|c||c@{\hspace{1.00ex}}%
                    c@{\hspace{0.50ex}}%
                    c@{\hspace{0.50ex}}%
                    c@{\hspace{0.50ex}}%
                    c@{\hspace{1.00ex}}%
                    c|}
\hline \raisebox{-1.25ex}{\large$r$} &
\multicolumn{6}{c|}{\hspace{3.00ex}\large$k$\Strut{2.50ex}{0ex}}
\\[-1.00ex]
& 6 & 5 &  4 & 3 & 2 & 1
\\
\hline\hline $6$ & $^q127^q$\Strut{2.50ex}{0ex} &
\multicolumn{5}{|c|}{}
\\
\cline{3-3} $5$ & $^q63^q$ & $^a2667^a$\Strut{2.25ex}{0ex} &
\multicolumn{4}{|c|}{}
\\
\cline{4-4} $4$ & $^q31^q$ & $^s468{-}519^r$ &
$^a11811^a$\Strut{2.25ex}{0ex} &\multicolumn{3}{|c|}{}
\\
\cline{5-5} $3$ & $^q15^q$ & $^e99^r$ & $^s839{-}970^r$ &
$^a11811^a$\Strut{2.25ex}{0ex} & \multicolumn{2}{|c|}{}
\\
\cline{6-6} $2$ & $^q7^q$ & $^s21^n$ & $^s77{-}93^r$ &
$^s381{-}396^f$ & $^a2667^a$\Strut{2.25ex}{0ex}
&\multicolumn{1}{|c|}{}
\\
\cline{7-7} $1$ & $^p3^p$ & $^p5^p$ & $^p9^p$ & $^p19^p$ & $^p43^p$ &
$^p127^p$\Strut{2.25ex}{0ex}
\\
\hline
\end{tabular}
\end{center}

\vspace{1.0cm}

\begin{center}
{\sf Bounds on $\cC_2(8,k,r)$}
\vspace{0.5cm}
\\[1.00ex]
\begin{tabular}{|c||c@{\hspace{1.00ex}}%
                    c@{\hspace{0.50ex}}%
                    c@{\hspace{0.50ex}}%
                    c@{\hspace{0.50ex}}%
                    c@{\hspace{0.50ex}}%
                    c@{\hspace{1.00ex}}%
                    c|}
\hline \raisebox{-1.25ex}{\large$r$} &
\multicolumn{7}{c|}{\large$k$\hspace{4.00ex}\Strut{2.50ex}{0ex}}
\\[-1.00ex]
& 7 & 6 & 5 &  4 & 3 & 2 & 1
\\
\hline\hline $7$ & $^q255^q$\Strut{2.50ex}{0ex} &
\multicolumn{6}{|c|}{}
\\
\cline{3-3} $6$ & $^q127^q$ & $^a10795^a$\Strut{2.25ex}{0ex} &
\multicolumn{5}{|c|}{}
\\
\cline{4-4} $5$ & $^q63^q$ & $^s1895{-}2139^r$ &
$^a97155^a$\Strut{2.25ex}{0ex} &\multicolumn{4}{|c|}{}
\\
\cline{5-5} $4$ & $^q31^q$ & $^s401{-}426^m$ & $^s6902{-}8279^r$ &
$^a200787^a$\Strut{2.25ex}{0ex} & \multicolumn{3}{|c|}{}
\\
\cline{6-6} $3$ & $^q15^q$ & $^s85^n$ & $^s634{-}843^r$ &
$^s6477{-}6897^g$ & $^a97155^a$\Strut{2.25ex}{0ex}
&\multicolumn{2}{|c|}{}
\\
\cline{7-7} $2$ & $^q7^q$ & $^s21^n$ & $^s75{-}93^\ell$ &
$^s323{-}346^c$ & $^s1567{-}1658^i$ & $^a10795^a$ \Strut{2.25ex}{0ex} &
\multicolumn{1}{|c|}{}
\\
\cline{8-8} $1$ & $^p3^p$ & $^p5^p$ & $^p9^p$ & $^p17^p$ & $^p37^p$ & $^p85^p$
& $^p255^p$\Strut{2.25ex}{0ex}
\\
\hline
\end{tabular}
\end{center}

\vspace{1.00cm}

\begin{center}
{\sf Bounds on $\cC_2(9,k,r)$}
\vspace{0.5cm}
\\[1.00ex]
\begin{tabular}{|c||c@{\hspace{1.00ex}}%
                    c@{\hspace{0.50ex}}%
                    c@{\hspace{0.50ex}}%
                    c@{\hspace{0.50ex}}%
                    c@{\hspace{0.50ex}}%
                    c@{\hspace{0.50ex}}%
                    c@{\hspace{1.00ex}}%
                    c|}
\hline \raisebox{-1.25ex}{\large$r$} &
\multicolumn{8}{c|}{\large$k$\hspace{4.00ex}\Strut{2.50ex}{0ex}}
\\[-1.00ex]
& 8 & 7 & 6 & 5 &  4 & 3 & 2 & 1
\\
\hline\hline $8$ & $^q511^q$\Strut{2.50ex}{0ex} &
\multicolumn{7}{|c|}{}
\\
\cline{3-3} $7$ & $^q255^q$ & $^a43435^a$ \Strut{2.25ex}{0ex} &
\multicolumn{6}{|c|}{}
\\
\cline{4-4} $6$ & $^q127^q$ & $^s7625{-}8683^r$ & $^a788035^a$
\Strut{2.25ex}{0ex} &\multicolumn{5}{|c|}{}
\\
\cline{5-5} $5$ & $^q63^q$ & $^s1614{-}1767^r$ & $^s55983{-}68371^r$ &
$^a3309747^a$\Strut{2.25ex}{0ex} &\multicolumn{4}{|c|}{}
\\
\cline{6-6} $4$ & $^q31^q$ & $^e371^r$ & $^s5143{-}7170^r$ &
$^d108574{-}118631^r$ & $^a3309747^a$\Strut{2.25ex}{0ex}
&\multicolumn{3}{|c|}{}
\\
\cline{7-7} $3$ & $^q15^q$ & $^s85^n$ & $^s609{-}829^r$ &
$^s5325{-}6379^r$ & $^s53383{-}59953^r$ & $^a788035^a$\Strut{2.25ex}{0ex}
& \multicolumn{2}{|c|}{}
\\
\cline{8-8} $2$ & $^q7^q$ & $^s21^n$ & $^s73^n$ & $^s281{-}346^\ell$ &
$^s1261{-}1325^i$ & $^s6205{-}6508^i$ & $^a43435^a$\Strut{2.25ex}{0ex}
&\multicolumn{1}{|c|}{}
\\
\cline{9-9} $1$ & $^p3^p$ & $^p5^p$ & $^p9^p$ & $^p17^p$ & $^p35^p$ & $^p73^p$
& $^p171^p$ & $^p511^p$\Strut{2.25ex}{0ex}
\\
\hline
\end{tabular}
\end{center}

\vspace{1.50ex}

\begin{center}
{\sf Bounds on $\cC_2(10,k,r)$}
\vspace{0.5cm}
\\[1.00ex]
\begin{scriptsize}
\begin{tabular}{|c||c@{\hspace{1.00ex}}%
                    c@{\hspace{0.50ex}}%
                    c@{\hspace{0.50ex}}%
                    c@{\hspace{0.50ex}}%
                    c@{\hspace{0.50ex}}%
                    c@{\hspace{0.50ex}}%
                    c@{\hspace{0.50ex}}%
                    c@{\hspace{1.00ex}}%
                    c|}
\hline \raisebox{-1.25ex}{\large$r$} &
\multicolumn{9}{c|}{\large$k$\hspace{4.00ex}\Strut{2.50ex}{0ex}}
\\[-1.00ex]
& 9 & 8 & 7 & 6 & 5 &  4 & 3 & 2 & 1
\\
\hline\hline $9$ & $^q1023^q$\Strut{2.50ex}{0ex} &
\multicolumn{8}{|c|}{}
\\
\cline{3-3} $8$ & $^q511^q$ & $^a174251^a$ \Strut{2.25ex}{0ex} &
\multicolumn{7}{|c|}{}
\\
\cline{4-4} $7$ & $^q255^q$ & $^s30590{-}34987^r$ & $^a6347715^a$
\Strut{2.25ex}{0ex} &\multicolumn{6}{|c|}{}
\\
\cline{5-5} $6$ & $^q127^q$ & $^s6475{-}7195^r$ &
$^d451631{-}555651^r$ & $^a53743987^a$\Strut{2.25ex}{0ex}
&\multicolumn{5}{|c|}{}
\\
\cline{6-6} $5$ & $^q63^q$ & $^s1489{-}1546^m$ & $^s41428{-}59127^r$ &
$^d1777360{-}1966467^r$ & $^a109221651^a$\Strut{2.25ex}{0ex}
&\multicolumn{4}{|c|}{}
\\
\cline{7-7} $4$ & $^q31^q$ & $^s341^n$ & $^s4906{-}7003^r$ &
$^s86468{-}109234^r$ & $^s1761639{-}1937127^r$ &
$^a53743987^a$\Strut{2.25ex}{0ex} & \multicolumn{3}{|c|}{}
\\
\cline{8-8} $3$ & $^q15^q$ & $^s85^n$ & $^s589{-}669^r$ &
$^s4563{-}6365^r$ & $^s41613{-}45230^i$ & $^s423181{-}476465^r$ &
$^a6347715^a$\Strut{2.25ex}{0ex} &\multicolumn{2}{|c|}{}
\\
\cline{9-9} $2$ & $^q7^q$ & $^s21^n$ & $^s73^n$ & $^s277{-}345^r$ &
$^s1155{-}1210^c$ & $^s4979{-}5197^i$ & $^s24991{-}26298^i$ &
$^a174251^a$\Strut{2.25ex}{0ex} &\multicolumn{1}{|c|}{}
\\
\cline{10-10} $1$ & $^p3^p$ & $^p5^p$ & $^p9^p$ & $^p17^p$ & $^p33^p$ &
$^p69^p$ & $^p147^p$ & $^p341^p$ & $^p1023^p$\Strut{2.25ex}{0ex}
\\
\hline
\end{tabular}
\end{scriptsize}
\end{center}

\noindent
\begin{itemize}
\item $a$ - all $\begin{tiny} \sbinomtwo{n}{k} \end{tiny}$ $k$-dimensional subspaces of $\F_2^n$.


\item $c$ - simple construction from $\CMRD$ (Theorem~\ref{thm:simpleCMRD}).

\item $d$ - de Caen Theorem (Theorem~\ref{thm:bound_k,k-1})


\item $e$ - a bound on the size of set of lines by Eisfeld and Metsch (Theorem~\ref{thm:EiMe}).

\item $f$ - Theorem~\ref{thm:b396}.

\item $g$ - Theorem~\ref{thm:b6897}.

\item $i$ - improved construction from $\CMRD$
(Theorems~\ref{thm:improved_CMRD} and~\ref{thm:improved_CMRDk_3}).

\item $\ell$ - lengthening theorem (Theorem~\ref{thm:lengthening})

\item $m$ - a set of lines in projective geometry by Metsch (Theorem~\ref{thm:Met03a}).

\item $n$ - $q$-covering design based on normal spreads (Theorem~\ref{thm:n_spreads}).

\item $p$ - covering of single points (Theorem~\ref{thm:optc2}).

\item $q$ - Theorem~\ref{thm:Turan=1}.

\item $r$ - recursive construction (Theorem~\ref{thm:recursiveC}).

\item $s$ - Sch\"{o}nheim bound (Theorem~\ref{thm:schonheim}).
\end{itemize}

\begin{theorem}
$\cC_2(10,5,3) \leq 45230$
\end{theorem}
\begin{proof}
To apply Theorem~\ref{thm:improved_CMRDk_3} we should consider the
structure of an appropriate $q$-covering design $\C_2(7,5,3)$. We start with
a $q$-covering design $\C_2(6,4,2)$ of size 21. This $q$-covering design
is obtained from the orthogonal complement of a normal 2-spread
in $\cG_2 (6,2)$, where the {\it orthogonal complement} $\dS^\perp$, of a set
$\dS$ of subspaces from $\F_q^n$, is defined by
$$
\dS^\perp \deff \bigl\{ X^\perp ~:~ X \in \dS \bigr\}~,
$$
where $X^\perp$ is the dual subspace of $X$. Hence, we can take a $q$-covering design
$\C_2(6,4,2)$ on $\F_2^6 = \V_0^{(6,1)} \cup \V_1^{(6,1)}$ of size
21 with exactly one subspace contained in $\V_0^{(6,1)}$. By applying
Construction~\ref{con:recursiveC} with this $q$-covering design $\C_2(6,4,2)$ and a
$q$-covering design $\C_2(6,5,3)$ of size 15, we
obtain a $q$-covering design $\C_2(7,5,3)$ on $\F_2^7 = \V_0^{(7,2)} \cup \V_1^{(7,2)}
\cup \V_2^{(7,2)} \cup \V_3^{(7,2)}$. In this design of size 99, there
is no element contained in $\V_0^{(7,2)}$, 80 elements meet
$\V_0^{(7,2)}$, $\V_1^{(7,2)}$, $\V_2^{(7,2)}$, and $\V_3^{(7,2)}$, each in at least one
vector, two elements lie in $\V_0^{(7,2)} \cup \V_2^{(7,2)}$,
two elements lie in $\V_0^{(7,2)} \cup \V_3^{(7,2)}$, and 15 elements
lie in $\V_0^{(7,2)} \cup \V_1^{(7,2)}$,

Applying the construction in Theorem~\ref{thm:improved_CMRDk_3}
with this $q$-covering design $\C_2(7,5,3)$,
yields a $q$-covering design $\C_2(10,5,3)$ of size $2^{15} +
\sbinomtwo{5}{2} \cdot 80 + 31 \cdot 2+ \cC_2 (5,5,3) =45231$.
But, the unique 5-dimensional subspace of the $q$-covering design $\C_2(5,5,3)$
can be omitted if we will make our choice of the other
subspaces in a similar way to the construction in
Theorem~\ref{thm:simpleCMRD}.
\end{proof}

\section{Conclusion and Problems for Future Research}
\label{sec:conclusion}

The minimal size, $\cC_q(n,k,r)$, of a $q$-covering design
$\C_q(n,k,r)$ was considered. A few techniques to obtain bounds on
$\cC_q(n,k,r)$ were discussed. Some of the results were given for
general $q$, while other were given only for $q=2$, even so some
of them can be definitely generalized. Tables with the lower and
upper bounds on $\cC_2(n,k,r)$, $1 \leq r \leq k <10$, were given.

There are many problems for future research and we will mention a
few. We will also propose some construction methods for which we
are unable to determine, at this point, how much successful they
would be. We would like to emphasize that a computer search which
will improve some of the specific results seems to be quite
difficult at this point of time.
\begin{enumerate}
\item We have given only new upper bounds on $\cC_q(n,k,r)$ with
related constructions. To reduce the gap between the lower and the
upper bounds also lower bounds should be considered. We can
suggest two methods to obtain new lower bounds. The first one is
to find analog theorems to the ones known for covering designs on
sets (as done in Theorem~\ref{thm:bound_k,k-1}~\cite{EtVa11a}).
A second method is to examine
the number of subspaces for each type, when sets of the form $\V_x^{(n,\ell)}$,
are considered. Inequalities related to the way that
$r$-dimensional subspaces are covered by $k$-dimensional subspaces
should be developed and solved to minimize the number of
$k$-dimensional subspaces in the $q$-covering design
$\C_q(n,k,r)$.

\item Theorems~\ref{thm:improved_CMRD}
and~\ref{thm:improved_CMRDk_3} can be further generalized to
obtain upper bounds $\cC_2(n,k,r)$, $r>3$. How the constructions
related to these theorems can be applied to obtain the best
bounds? How can a general bound for all $r \geq 2$ be
formulated? Can we obtain better $q$-covering designs, to be applied for the
recursion in the proofs of these theorems?

\item Usually, the upper bound implied by
Theorem~\ref{thm:lengthening} can be improved by a better
construction. In some cases we couldn't produce a better bound or
we only obtained a minor improvement. Can a more significant
improvement be made in these cases ($\cC_2(8,5,2)$,
$\cC_2(9,5,2)$, $\cC_2(10,6,2)$).

\item The constructions in which a subspace transversal design
based on lifting of an MRD code seems to be very powerful in
obtaining good bounds on $\cC_2(n,k,r)$. Unfortunately, a subspace
transversal design exists if and only if $k \leq n-k$. Therefore,
we ask the following question. Given $k > n-k$ and the related
sets $\V_0^{(n,k)}$, $\V_1^{(n,k)} , \ldots$, what is the
size of the smallest set $\dS$ of $k$-dimensional subspaces such
that for each $X \in \dS$, $X$ meets each $\V_y^{(n,k)}$, $y \neq 0$,
in exactly one point; and each $r$-dimensional subspace of
$\F_q^n$ which meets each $\V_y^{(n,k)}$ in at most one point, is contained
in at least one $k$-dimensional subspace of $\dS$?

\item Let $\{ X_1$, $X_2$,...,$X_t \}$ be a set of $k$-dimensional
subspaces of $\F_q^n$ such that $Y = X_i \cap X_j$ for each $1
\leq i < j \leq t$ and for each one-dimensional subspace $Z \in
\cG_q(n,1)$ there exists at least one $i$, $1 \leq i \leq t$, such
that $Z \subset X_i$. Let $X_i = Y \cup U_i$, such that $Y \cap
U_i = \varnothing$. What is the smallest set $\dS$ of
$k$-dimensional subspaces such that for each $X \in \dS$, $X$
meets each $U_i$, in at most one point; and each
two-dimensional subspace of~$\F_q^n$, disjoint from $Y$, which
meets each $U_i$ in at most one point is contained in at least one
$k$-dimensional subspace of $\dS$? This is somewhat a
generalization of a subspace transversal design. Related question
for general $r$-dimensional subspaces instead of two-dimensional
subspaces is also of interest.

\item What is the best upper bound on the density of a
$q$-covering design $\cC_q(n,k,r)$ which can be obtained for any
given $k$ and $r$ such that $1 < r < k < n-1$, when $n$ tends to
infinity? What is the best bound for a given $r$ and all $k >r$?
What is the best bound for a given $k$ and all $r < k$?
\end{enumerate}

\begin{center}
{\bf Acknowledgments}
\end{center}
In the paper we consider a concept which is of interest in coding theory,
design, $q$-analogs, and projective geometry. This made some difficulties
in choosing and using the right notations, such that all communities will be
satisfied. The author is grateful to one of the anonymous reviewers which has done
a remarkable job by suggesting the right notations and pointing on places
which might be confusing. Moreover, he has spotted some errors and
suggested how to amend the
papers in the right places, and most of his suggestions were accepted.
The other anonymous reviewers should also be acknowledged for providing some important comments.
The author also thanks Ameerah N. Chowdhury for providing the paper
of Klaus Metsch~\cite{Met03} and paying his attention to some of
the covering problems which were considered in terms
of projective geometries. He also thanks Natalia Silberstein for
commenting on an early version of this paper.


\end{document}